\documentclass{article}%
\usepackage{graphicx}
\usepackage{amsmath}
\usepackage{amsfonts}
\usepackage{amssymb}%
\providecommand{\U}[1]{\protect\rule{.1in}{.1in}}
\newtheorem{theorem}{Theorem}
{}

\newtheorem{corollary}{Corollary}

\newtheorem{lemma}{Lemma}
{}

\newtheorem{proposition}{Proposition}

\newenvironment{proof}[1][Proof]{\textbf{#1.} }{\ \rule{0.5em}{0.5em}}

\begin{document}

\title{On\ the Schr\"{o}dinger Operator with a Periodic PT-symmetric Matrix Potential}
\author{O. A. Veliev\\{\small Dogus University, \ Istanbul, Turkey.}\\\ {\small e-mail: oveliev@dogus.edu.tr}}
\date{}
\maketitle

\begin{abstract}
In this article we obtain asymptotic formulas for the Bloch eigenvalues of the
operator $L$ generated by a system of Schr\"{o}dinger equations with periodic
PT-symmetric complex-valued coefficients. Then using these formulas we
classify the spectrum $\sigma(L)$ of $L$ and find a condition on the
coefficients for which $\sigma(L)$ contains all half line $[H,\infty)$ for
some $H.$

Key Words: Non-self-adjoint differential operator, PT-symmetric coefficients,
Periodic matrix potential.

AMS Mathematics Subject Classification: 34L05, 34L20.

\end{abstract}

\section{Introduction and Preliminary Facts}

Let $L(Q)$ be the differential operator generated in the space $L_{2}%
^{m}(-\infty,\infty)$ of the vector functions by the differential expression
\begin{equation}
-y^{^{\prime\prime}}+Qy, \tag{1}%
\end{equation}
where $Q=\left(  q_{i,j}\right)  $ is a $m\times m$ matrix with the
PT-symmetric $\pi$-periodic locally square integrable entries $q_{i,j}.$ In
other words,%
\begin{equation}
\overline{q_{i,j}\left(  -x\right)  }=q_{i,j}\left(  x\right)  ,\text{
}q_{i,j}\left(  x+\pi\right)  =q_{i,j}\left(  x\right)  ,\text{ }q_{i,j}\in
L_{2}\left[  0,\pi\right]  . \tag{2}%
\end{equation}
It is well-known that [1, 5, 7] the spectrum $\sigma(L(Q))$ of the operator
$L(Q)$ is the union of the spectra $\sigma(L_{t}(Q))$ of the operators
$L_{t}(Q)$ for $t\in(-1,1]$ generated in $L_{2}^{m}[0,\pi]$ by the
differential expression (1) and the quasiperiodic conditions
\begin{equation}
y^{^{\prime}}\left(  \pi\right)  =e^{i\pi t}y^{^{\prime}}\left(  0\right)
,\text{ }y\left(  \pi\right)  =e^{i\pi t}y\left(  0\right)  . \tag{3}%
\end{equation}
Note that $L_{2}^{m}(a,b)$ is the set of the vector functions $f=\left(
f_{1},f_{2},...,f_{m}\right)  $ with $f_{k}\in L_{2}(a,b)$ for $k=1,2,...,m.$
The norm $\left\Vert \cdot\right\Vert $ and inner product $(\cdot,\cdot)$ in
$L_{2}^{m}(a,b)$ are defined by%
\[
\left\Vert f\right\Vert =\left(  \int\nolimits_{(a,b)}^{{}}\left\vert f\left(
x\right)  \right\vert ^{2}dx\right)  ^{\frac{1}{2}},\text{ }(f,g)=\int
\nolimits_{(a,b)}^{{}}\left\langle f\left(  x\right)  ,g\left(  x\right)
\right\rangle dx,
\]
where $\left\vert \cdot\right\vert $ and $\left\langle \cdot,\cdot
\right\rangle $ are the norm and inner product in $\mathbb{C}^{m}.$ For
$t\in(-1,1]$ the spectra $\sigma(L_{t}(Q))$ of the operators $L_{t}(Q)$
consist of the eigenvalues called the Bloch eigenvalues of $L(Q)$. Any
eigenfunction $\Psi_{\lambda(t)}$ corresponding to the Bloch eigenvalue
$\lambda(t)$ is called the Bloch function. We say that $\Psi_{\lambda(t)}$ is
a normalized Bloch Function if its $L_{2}^{m}(0,\pi)$ norm is $1$.

Let us introduce some preliminary results and describe briefly the scheme of
the paper. The eigenvalues of the operator $L_{t}(Q)$ are the roots of the
characteristic determinant
\begin{equation}
\Delta(\lambda,t)=\det(Y_{j}^{(\nu-1)}(\pi,\lambda)-e^{i\pi t}Y_{j}^{(\nu
-1)}(0,\lambda))_{j,\nu=1}^{2}= \tag{4}%
\end{equation}%
\[
e^{i2m\pi t}+f_{1}(\lambda)e^{i(2m-1)\pi t}+f_{2}(\lambda)e^{i(2m-2)\pi
t}+...+f_{2m-1}(\lambda)e^{i\pi t}+1
\]
which is a polynomial of $e^{i\pi t}$\ with entire coefficients $f_{1}%
(\lambda),f_{2}(\lambda),...$, where $Y_{1}(x,\lambda)$ and $Y_{2}(x,\lambda)$
are the solutions of the matrix equation
\[
-Y^{^{\prime\prime}}(x)+Q\left(  x\right)  Y(x)=\lambda Y(x)
\]
satisfying $Y_{1}(0,\lambda)=O_{m}$, $Y_{1}^{^{\prime}}(0,\lambda)=I_{m}$ and
$Y_{2}(0,\lambda)=I_{m}$, $Y_{2}^{^{\prime}}(0,\lambda)=O_{m}$ (see [6]
Chapter 3). Here $O_{m}$ and $I_{m}$ are $m\times m$ zero and identity
matrices respectively. It is clear that

$\varphi_{k,1,t}=\left(
\begin{array}
[c]{c}%
\frac{e^{i\left(  2k+t\right)  x}}{\sqrt{\pi}}\\
0\\
\vdots\\
0
\end{array}
\right)  ,$ $\varphi_{k,2,t}=\left(
\begin{array}
[c]{c}%
0\\
\frac{e^{i\left(  2k+t\right)  x}}{\sqrt{\pi}}\\
\vdots\\
0
\end{array}
\right)  ,...,\varphi_{k,m,t}=\left(
\begin{array}
[c]{c}%
0\\
\vdots\\
0\\
\frac{e^{i\left(  2k+t\right)  x}}{\sqrt{\pi}}%
\end{array}
\right)  $ are the normalized eigenfunctions of the operator $L_{t}(O_{m})$
corresponding to the eigenvalue $\left(  2k+t\right)  ^{2}$. If $t\neq0,1,$
then the multiplicity of the eigenvalue $\left(  2k+t\right)  ^{2}$ is $m$ and
the corresponding eigenspace is $E_{k}(t)=span\left\{  \varphi_{k,1,t}%
,\varphi_{k,2,t},...,\varphi_{k,m,t}\right\}  $. In the cases $t=0$ and $t=1$
the multiplicity of the nonzero eigenvalues $\left(  2k\right)  ^{2}$ and
$\left(  2k+1\right)  ^{2}$ is $2m$ and the corresponding eigenspaces are

$E_{k}(0)=span\left\{  \varphi_{n,j,0}:n=k,-k;\text{ }j=1,2,...m\right\}  $ and

$E_{k}(\pi)=span\left\{  \varphi_{n,j,1}:n=k,-(k+1);\text{ }%
j=1,2,...m\right\}  $ respectively.

The brief scheme of the paper is the following. First we note that (see
Theorem 1) if $\lambda$ is an eigenvalue of multiplicity $p$ of the operator
$L_{t}(Q),$ then $\overline{\lambda}$ is also an eigenvalue of the same
multiplicity of $L_{t}(Q).$ This is the characteristic property of the
Schr\"{o}dinger operator with PT-symmetric potential. Then, in Theorem 2 we
prove that there exists a constant $c$ such that the eigenvalues $\lambda(t)$
of the operator $L_{t}(Q)$ lie on the $c$ neighborhoods of the eigenvalues
$(2n+t)^{2}$ of $L_{t}(O_{m})$ for $n\in\mathbb{Z}$. For this we use the
formula
\begin{equation}
\left(  \lambda(t)-\left(  2n+t\right)  ^{2}\right)  \left(  \Psi_{\lambda
(t)},\varphi_{n,s,t}\right)  =\left(  Q\Psi_{\lambda(t)},\varphi
_{n,s,t}\right)  , \tag{5}%
\end{equation}
which can be obtained from the equality $L_{t}\left(  Q\right)  \Psi
_{\lambda(t)}=\lambda(t)\Psi_{\lambda(t)}$ by multiplying both sides by
$\varphi_{n,s,t}(x)$ and using $L_{t}\left(  O_{m}\right)  \varphi
_{n,s,t}(x)=\left(  2n+t\right)  ^{2}\varphi_{n,s,t}(x).$ In the other words,
first we consider the operator $L_{t}(O_{m})$ for an unperturbed operator and
the operator of multiplication by $Q$ for a perturbation.

Then to obtain a sharp asymptotic formulas we consider the operator $L_{t}(Q)$
as perturbation of $L_{t}(A),$ where
\begin{equation}
A=\int\nolimits_{(0,\pi)}Q\left(  x\right)  dx, \tag{6}%
\end{equation}
\ by $Q-A,$ that is, we take the operator $L_{t}(A)$ for an unperturbed
operator and the operator of multiplication by $Q-A$ for a perturbation.
Therefore first we analyze the eigenvalues and eigenfunctions of $L_{t}(A)$.

Using (2) and the substitution $t=-x$ one can get the equality%
\[
\overline{\int_{0}^{\pi}q_{i,j}\left(  x\right)  dx}=\int_{0}^{\pi}%
q_{i,j}\left(  -x\right)  dx=-\int_{0}^{-\pi}q_{i,j}\left(  t\right)
dt=\int_{0}^{\pi}q_{i,j}\left(  t\right)  dt
\]
which means that%
\[
\int_{0}^{\pi}q_{i,j}\left(  x\right)  dx\in\mathbb{R}%
\]
for all $i$ and $j$. Hence the entries of the matrix $A$ are the real numbers.
Therefore, the eigenvalues of the matrix $A$ consist of the real eigenvalues
and the pairs of the conjugate complex numbers. The distinct eigenvalues of
$A$ are denoted by $\mu_{1},\mu_{2},...,\mu_{p}.$ Without loss of generality
we denote the real distinct eigenvalues by $\mu_{1}<\mu_{2}<...<\mu_{s}$ and
the nonreal distinct eigenvalues by$\ \mu_{s+1},\mu_{s+2},...,\mu_{p}.$ It
readily implies that the spectrum of $L(A)$ consists of the real half line
$[\mu_{1},\infty)$ and nonreal half lines
\begin{equation}
\lbrack\mu_{j}+a:\text{ }a\in\lbrack0,\infty) \tag{7}%
\end{equation}
for $j=s+1,s+2,...,p.$

In Theorem 3 we find a sharp and uniform, with respect to the quasimomenta
$t\in(-1,1],$ asymptotic formula for the Bloch eigenvalues of $L(Q)$ in term
of the eigenvalues of the matrix $A$ for any matrix potential $Q$ with locally
square integrable entries. This asymptotic formula implies that the spectrum
of $L(Q)$ is asymptotically close to the spectrum of the operator $L(A).$
Therefore, if the matrix $A$ has no real eigenvalues then the spectrum of the
operator $L(Q)$ \ with general potential $Q$ approaches the nonreal half lines
(7), which implies that the real component $\sigma(L(Q))\cap\mathbb{R}$ of the
spectrum $\sigma(L(Q))$ is contained in a finite interval $[a,b]$ (see Theorem
4). Then we prove that, if the entries of $Q$ are the PT-symmetric functions,
and the matrix $A$ has a real eigenvalue with odd multiplicity, then
$\sigma(L(Q))$ contains the main part (in sense of (41)) of $[0,\infty)$ (see
Theorem 6). For this we consider the multiplicity of the Bloch eigenvalues
(see Theorem 5). Note that if $m$ is an odd number, then the matrix $A$ has a
real eigenvalue with odd multiplicity, because the total sum of the
multiplicities of the eigenvalues $\mu_{1},\mu_{2},...,\mu_{p}$ is $m$ and the
nonreal eigenvalues are the pairs of the conjugate complex numbers with the
same multiplicity. Therefore Theorem 6 implies that if $m$ is an odd number
and (2) holds, then the spectrum of $L(Q)$ contains the main part of
$[0,\infty)$ (see Corollary 1). In Theorem 7 we find a condition on the real
eigenvalues of $A$ for which the spectrum of $L(Q)$ contains all half line
$[H,\infty)$ for some $H$.

Finally, note that in the papers (see [4, 9, 11]) we investigated the
non-self-adjoint Schr\"{o}dinger operator with a periodic matrix potential.
However, as far as I know, this paper is the first paper about the
Schr\"{o}dinger operator with a PT-symmetric periodic matrix potential. In
this paper, we investigate the influence of PT-symmetry on spectrum of the
operator $L(Q)$.

\section{On the Bloch eigenvalues and Spectrum}

First let us note that in the case $m=1,$ it is well-known that if $\lambda\in
L_{t}(Q),$ then $\overline{\lambda}\in L_{t}(Q)$ (see [3, 8, 12]). Taking into
account that this obvious fact was not formulated for arbitrary $m$, we prove
the following more general statement.

\begin{theorem}
If $\lambda$ is an eigenvalue of multiplicity $v$ of the operator $L_{t}(Q),$
then $\overline{\lambda}$ is also an eigenvalue of the same multiplicity of
$L_{t}(Q).$
\end{theorem}

\begin{proof}
First let us note that $\lambda\in\sigma\left(  L_{t}(Q)\right)  $ if and only
if there exists a solution $\Psi(\cdot,\lambda)$ of the equation
\begin{equation}
-y^{^{\prime\prime}}(x)+Q\left(  x\right)  y(x)=\lambda y(x) \tag{8}%
\end{equation}
satisfying the equality
\begin{equation}
\Psi(x+\pi,\lambda)=e^{it}\Psi(x,\lambda) \tag{9}%
\end{equation}
for all $x\in(-\infty,\infty).$ Indeed if $\lambda\in\sigma\left(
L_{t}(Q)\right)  ,$ then there exists a solution $\Phi(x,\lambda)$ of (8)
satisfying (3). Therefore the solution $y(x,\lambda)$ of (8) defined by
\[
y(x,\lambda)=\Psi(x+\pi,\lambda)-e^{it}\Psi(x,\lambda)
\]
satisfies the initial conditions $y(0,\lambda)=y^{\prime}(0,\lambda)=0.$ Then
by the uniqueness theorem $y(x,\lambda)=0$ for all $x\in(-\infty,\infty),$
that is, (9) holds.

Now suppose that there exists a solution $\Psi(x,\lambda)$ of (8) satisfying
(9). Then it is clear that $\Psi(x,\lambda)$ satisfies (3). It means that
$\lambda\in\sigma\left(  L_{t}(Q)\right)  .$

Now we are ready to prove the theorem. If $\lambda$ is an eigenvalue of
$L_{t}(Q),$ then there exists a solution $\Psi(x,\lambda)$ of (8) satisfying
(9). Then using (2) one can readily see that the function $\Phi(x)$ defined by
$\Phi(x,\lambda)=\overline{\Psi(-x,\lambda)}$ satisfies the equation
\[
-y^{^{\prime\prime}}(x)+Q(x)y(x)=\overline{\lambda}y(x)
\]
and the equality
\[
\Phi(x+\pi)=\overline{\Psi(-x-\pi,\lambda)}=\overline{e^{-it}\Psi(-x,\lambda
)}=e^{it}\Phi(x).
\]
It means that $\overline{\lambda}$ is also an eigenvalue of $L_{t}(Q)$ and
$\Phi$ is the corresponding eigenfunction.

Instead of the equation (8) using the equations
\begin{equation}
-\Psi_{j}^{^{\prime\prime}}(x,\lambda)+Q\left(  x\right)  \Psi_{j}%
(x,\lambda)-\lambda\Psi_{j}(x,\lambda)=\Psi_{j-1}(x,\lambda) \tag{10}%
\end{equation}
for the associated functions $\Psi_{1}(x,\lambda),\Psi_{2}(x,\lambda
),...,\Psi_{s}(x,\lambda),$ where $j=1,2,...,s$ and $\Psi_{0}(x,\lambda
)=\Psi(x,\lambda)$ and repeating the above argument we conclude that if
$y(x,\lambda)$ is an associated function corresponding to the eigenfunction
$\Psi(x,\lambda),$ then $\overline{y(-x,\lambda)}$ \ is an associated
eigenfunction corresponding to the eigenfunction $\Phi(x,\lambda).$ Therefore
the multiplicities of the eigenvalues $\lambda$ and $\overline{\lambda}$ are
the same. The theorem is proved.
\end{proof}

Now we estimate the Bloch eigenvalues. First we prove the following simple
estimation for the Bloch eigenvalues.

\begin{theorem}
$(a)$ There exists a constant $M$ such that
\begin{equation}
\sup_{\lambda\in\sigma(L(Q)),\text{ }x\in\lbrack0,\pi]}\left\vert
\Psi_{\lambda}(x)\right\vert <M, \tag{11}%
\end{equation}
where $\Psi_{\lambda}$ is a normalized Bloch function corresponding to the
eigenvalue $\lambda.$

$(b)$ All eigenvalues of the operator $L_{t}(Q)$ lie on the union of the
disks
\begin{equation}
D\left(  (2n+t)^{2},c\right)  :=\{\lambda\in\mathbb{C}:\left\vert
\lambda-(2n+t)^{2}\right\vert <c\} \tag{12}%
\end{equation}
for $n\in\mathbb{Z},$ where $t\in(-1,1]$, $c=MB$ and
\begin{equation}
B=\left(  \int\limits_{0}^{\pi}\left(  \sum_{i,j=1}^{m}\left\vert
q_{i,j}\left(  x\right)  \right\vert ^{2}\right)  dx\right)  ^{1/2}. \tag{13}%
\end{equation}

\end{theorem}

\begin{proof}
$(a)$ Let $\Psi_{\lambda(t)}$ be a normalized eigenfunction (Bloch function)
corresponding to the eigenvalue $\lambda(t)$ of the operator $L_{t}(Q).$ There
exists $n$ such that $(2n-1)^{2}\leq\operatorname{Re}\lambda(t)<(2n+1)^{2}.$
Then it is clear that there exists a constant $C$ such that
\begin{equation}
\sum_{k\in\left(  \mathbb{Z}\backslash B(n)\right)  }\frac{1}{\left\vert
\lambda(t)-(2k+t)^{2}\right\vert }<C \tag{14}%
\end{equation}
for all $t\in(-1,1],$ where $B(n)=\left\{  \pm\left(  n-1\right)  ,\pm
n,\pm\left(  n+1\right)  \right\}  .$ In the decomposition
\[
\Psi_{\lambda(t)}=\sum_{\substack{k\in B(n),\\s=1,2,...,m}}\left(
\Psi_{\lambda(t)},\varphi_{k,s,t}\right)  \varphi_{k,s,t}+\sum_{\substack{k\in
\left(  \mathbb{Z}\backslash B(n)\right)  ,\\s=1,2,...,m}}\left(
\Psi_{\lambda(t)},\varphi_{k,s,t}\right)  \varphi_{k,s,t}%
\]
of $\Psi_{\lambda(t)}$ by the orthonormal basis
\begin{equation}
\left\{  \varphi_{k,s,t}:k\in\mathbb{Z}\text{ },\text{ }s=1,2,...,m\right\}
\tag{15}%
\end{equation}
replacing $\left(  \Psi_{\lambda(t)},\varphi_{k,s,t}\right)  $ for
$k\in\left(  \mathbb{Z}\backslash B(n)\right)  $ by
\[
\frac{\left(  \Psi_{\lambda(t)},Q^{\ast}\varphi_{k,s,t}\right)  }%
{\lambda(t)-(2k+t)^{2}}%
\]
(see (5)) and then using (14), (13) and the obvious inequalities
\[
\left\vert \left(  \Psi_{\lambda(t)},Q^{\ast}\varphi_{k,s,t}\right)
\right\vert \leq\left\Vert Q^{\ast}\varphi_{k,s,t}\right\Vert \leq\frac
{1}{\sqrt{\pi}}B
\]
we obtain the proof of (11).

$(b)$ Suppose that there exists an eigenvalue $\lambda(t)$ of $L_{t}(Q)$ lying
out of $D(\left(  2n+t)^{2},c\right)  $ for all $n\in\mathbb{Z}$. Then the
inequality%
\[
\left\vert \lambda(t)-(2n+t)^{2}\right\vert \geq c
\]
holds for all $n\in\mathbb{Z}$. Using the Parseval's equality for the
orthonormal basis (15) and then (5) we get
\begin{equation}
1=\sum_{\substack{n\in\mathbb{Z},\text{ }\\s=1,2,...,m}}\left\vert \left(
\Psi_{\lambda(t)},\varphi_{n,s,t}\right)  \right\vert ^{2}\leq\sum
_{\substack{n\in\mathbb{Z},\text{ }\\s=1,2,...,m}}\frac{\left\vert \left(
Q\Psi_{\lambda(t)},\varphi_{n,s,t}\right)  \right\vert ^{2}}{c^{2}}%
=\frac{\left\Vert Q\Psi_{\lambda(t)}\right\Vert ^{2}}{c^{2}}, \tag{16}%
\end{equation}
where
\[
\left\Vert Q\Psi_{\lambda(t)}\right\Vert ^{2}=\int\limits_{0}^{\pi}\left\vert
Q(x)\Psi_{\lambda(t)}(x)\right\vert ^{2}dx.
\]
On the other hand, it follows from (11) that
\[
\left\vert Q(x)\Psi_{\lambda(t)}(x)\right\vert ^{2}<\left(  \sum_{i,j=1}%
^{m}\left\vert q_{i,j}\left(  x\right)  \right\vert ^{2}\right)  M^{2}.
\]
Therefore by (13) we have
\begin{equation}
\left\Vert Q\Psi_{\lambda(t)}\right\Vert ^{2}<c^{2}. \tag{17}%
\end{equation}
Using (17) in (16) we get a contradiction $1<1.$ The theorem is proved.
\end{proof}

Now to obtain a sharp and uniform with respect to $t\in(-1,1]$ asymptotic
formula we consider the operator $L_{t}(Q)$ as perturbation of $L_{t}(A)$ and
analyze the spectrum of the operators $L_{t}(A),$ where $A$ is defined in (6).
For this we introduce the following notations. Suppose the matrix $A$ has $p$
distinct eigenvalues $\mu_{1},\mu_{2},...,\mu_{p}$ with multiplicities
$m_{1},m_{2},...,m_{p}$ respectively, where $m_{1}+m_{2}+...+m_{p}=m$. Let
$u_{j,1},$ $u_{j,2},...u_{j,s_{j}}$ be the linearly independent eigenvectors
corresponding to the eigenvalue $\mu_{j}.$ Denote by $u_{j,s,1},$
$u_{j,s,2},...u_{j,s,r_{j,s}-1}$ the associated vectors corresponding to the
eigenvector $u_{j,s},$ such that
\[
\left(  A-\mu_{j}I\right)  u_{j,s,k}=u_{j,s,k-1}%
\]
for $k=1,2,...,r_{j,s}-1,$ where $u_{j,s,0}=u_{j,s}.$ Note that $r_{j,s}$ is
called the multiplicity of the eigenfunction $u_{j,s}$ and $r_{j,1}%
+r_{j,2}+...+r_{j,s_{j}}=m_{j}.$ The number $r_{j}$ defined by $r_{j}=\max
_{s}r_{j,s}$ is a maximum multiplicity of the eigenfunctions corresponding to
the eigenvalue $\mu_{j}.$ It is not hard to see that the eigenvalues,
eigenfunctions and associated functions of $L_{t}(A)$ are%

\begin{equation}
\mu_{k,j}(t)=\left(  2k+t\right)  ^{2}+\mu_{j},\text{ }\Phi_{k,j,s}%
(x)=u_{j,s}e^{i\left(  2k+t\right)  x},\text{ }\Phi_{k,j,s,r}(x)=u_{j,s,r}%
e^{i\left(  2k+t\right)  x} \tag{18}%
\end{equation}
respectively, since they satisfy the equation obtained from (10) by replacing
$Q$ to $A.$ Similarly, the eigenvalues, eigenfunctions, and associated
functions of $L_{t}^{\ast}(A)$ are $\overline{\mu_{k,j}},$
\begin{equation}
\Phi_{k,j,s}^{\ast}(x)=u_{j,s}^{\ast}e^{i\left(  2k+t\right)  x}\text{ }%
\And\Phi_{k,j,s,r}^{\ast}(x)=u_{j,s,r}^{\ast}e^{i\left(  2k+t\right)  x},
\tag{19}%
\end{equation}
where $u_{j,s}^{\ast}=$ $u_{j,s,0}^{\ast}$ and $u_{j,s,r}^{\ast}$ are the
eigenvector and associated vector of \ $A^{\ast}$ corresponding to
$\overline{\mu_{j}}.$ To obtain the sharp asymptotic formulas we use the
formula
\begin{equation}
(\lambda-\mu_{n,j})^{r+1}(\Psi,\Phi_{n,j,s,r}^{\ast})=\sum_{q=0}^{r}%
(\lambda-\mu_{n,j})^{q}((Q-A)\Psi,\Phi_{n,j,s,q}^{\ast}) \tag{20}%
\end{equation}
of [9] (see (15) of [9]), where $\Psi$ is a normalized eigenfunction of
$L_{t}(Q)$ corresponding to the eigenvalue $\lambda$ and estimate the
multiplicands $(\Psi,\Phi_{n,j,s,r}^{\ast})$ and $((Q-A)\Psi,\Phi
_{n,j,s,q}^{\ast})$ in the left-hand side and the right-hand side of (20). For
this we use the following proposition which can be proved by direct
calculations. Note that (20) was obtained from the equality $L_{t}\left(
Q\right)  \Psi=\lambda\Psi$ by multiplying both sides by $\Phi_{n,j,s,r}%
^{\ast}.$

\begin{proposition}
Let $A(k,t)$ be $\left\{  k\right\}  ,$ $\left\{  \pm k\right\}  ,$ $\left\{
k,-k-1\right\}  $ and $\left\{  k,-k+1\right\}  $ respectively if
\ $t\in\left(  (-2/3,-1/3)\cup(1/3,2/3)\right)  ,$ $t\in(-1/3,1/3),$
$t\in(2/3,1)$ and $t\in(-1,-2/3).$ If $\left\vert k\right\vert >1,$ then
\[
\left\vert (2n+t)^{2}-(2k+t)^{2}\right\vert \geq\frac{4}{3}\left(  2\left\vert
k\right\vert -1\right)
\]
for all $n\in\left(  \mathbb{Z}\backslash A(k,t)\right)  .$
\end{proposition}

In the following estimations we use the positive constants denoted by $c_{k}$
for $k=1,2,...$, independent on $t,$ whose exact values are inessential.

\begin{lemma}
Let $\Psi$ be a normalized eigenfunction of $L_{t}(Q)$ corresponding to the
eigenvalue $\lambda$ lying in the disk $D(\left(  2k+t)^{2},c\right)  ,$ where%
\begin{equation}
\left\vert k\right\vert >1,\frac{4}{3}\left(  2\left\vert k\right\vert
-1\right)  >3c \tag{21}%
\end{equation}
and $c$ is defined in (12). Then there exists $n\in A(k,t)$ such that
\begin{equation}
\left\vert (\Psi,\Phi_{n,j,s,r}^{\ast})\right\vert \geq c_{1} \tag{22}%
\end{equation}
for some $j,s,r$ and
\begin{equation}
\left\vert ((Q-A)\Psi,\Phi_{n,j,s,q}^{\ast})\right\vert \leq c_{2}\left(
\frac{1}{|k|}+q_{k}\right)  , \tag{23}%
\end{equation}
for all $j,s,q$, where
\[
q_{k}=\max\{\mid q_{i,j,s}\mid:i,j=1,2,...m;\text{ }s\in\left\{  \pm
2k,\pm(2k+1),\pm(2k-1)\right\}  ,
\]
$\ $%
\[
q_{i,j,s}=\frac{1}{\sqrt{\pi}}\int\limits_{0}^{\pi}q_{i,j}\left(  x\right)
e^{-2isx}dx,
\]
and $q_{i,j}(x)$ is the entry of the matrix $Q(x)$.
\end{lemma}

\begin{proof}
By Proposition 1 we have
\begin{equation}
|\lambda-\left(  2n+t\right)  ^{2}|\geq\frac{4}{3}\left(  2\left\vert
k\right\vert -1\right)  -c>2c \tag{24}%
\end{equation}
for all $t\in(-1,1]$ and $n\in\left(  \mathbb{Z}\backslash A(k,t)\right)  $ if
$\lambda\in D(\left(  2k+t)^{2},c\right)  $ and $k$ satisfies conditions (21).
Therefore using (5) Bessel inequality and (17) we obtain
\[
\sum\limits_{\substack{n\in\left(  \mathbb{Z}\backslash A(k,t)\right)
\\s=1,2,...,m}}\left\vert \left(  \Psi,\varphi_{n,s,t}\right)  \right\vert
^{2}=\sum\limits_{\substack{n\in\left(  \mathbb{Z}\backslash A(k,t)\right)
\\s=1,2,...,m}}\frac{\left\vert \left(  \Psi Q,\varphi_{n,s,t}\right)
\right\vert }{\left\vert \lambda-\left(  2\pi n+t\right)  ^{2}\right\vert
}\leq\frac{\left\Vert \Psi Q\right\Vert ^{2}}{4c^{2}}\leq\frac{1}{4}.
\]
It with the Parsaval's equality implies that there exist $n\in A(k,t)$ and
$i\in\left\{  1,2,...,m\right\}  $ such that
\[
\left\vert \left(  \Psi,\varphi_{n,i,t}\right)  \right\vert >c_{3}.
\]
Since the system of the root vectors of the matrix $A^{\ast}$ is a basis of
$\mathbb{C}^{m}$ and (19) holds, the last inequality implies that (22) holds
for some $j,s,r$.

Now to prove (23) we estimate the term $((Q-A)\Psi,\varphi_{n,i,t})$\ and take
(19) into account. Using the decomposition
\[
\Psi=\sum\limits_{v\in\mathbb{Z},\text{ }s=1,2,...,m}\left(  \Psi
,\varphi_{v,s,t}\right)  \varphi_{v,s,t}%
\]
of $\Psi$ by the orthonormal basis (15) we obtain.
\begin{equation}
\left(  (Q-A)\Psi,\varphi_{n,i,t}\right)  =\sum\limits_{\substack{v\in\left(
\mathbb{Z}\backslash\left\{  n\right\}  \right)  ,\\\text{ }s=1,2,...,m}%
}q_{i,s,n-v}\left(  \Psi,\varphi_{v,s,t}\right)  . \tag{25}%
\end{equation}
The right-hand side of (25) is the sum of
\[
S_{1}=:\sum_{v\in A(k,t)\backslash\left\{  n\right\}  ;\text{ }s=1,2,...,m}%
q_{i,s,n-v}\left(  \Psi,\varphi_{v,s,t}\right)
\]
and%
\[
S_{2}=:\sum\limits_{v\in\left(  \mathbb{Z}\backslash A(k,t)\right)  ;\text{
}s=1,2,...,m}q_{i,s,n-v}\left(  \Psi,\varphi_{v,s,t}\right)  .
\]
Here $\Psi$ and $\varphi_{v,s,t}$ are the normalized eigenfunctions and the
set $A(k,t)\backslash\left\{  n\right\}  $ consist of at most one number
Moreover, it follows from the definition of $A(k,t)$ that
\[
(n-v)\in\left\{  \pm2k,\pm(2k+1),\pm(2k-1)\right\}
\]
for all $n\in A(k,t)$ and $v\in\left(  A(k,t)\backslash\left\{  n\right\}
\right)  .$ Therefore, using the definition of $q_{k}$ we obtain
\begin{equation}
\left\vert S_{1}\right\vert \leq mq_{k}. \tag{26}%
\end{equation}
Now let us estimate $S_{2}.$ By (5) and (24)\ we have%
\[
\left\vert S_{2}\right\vert \leq c_{4}\frac{1}{|k|}\sum\limits_{\substack{v\in
\left(  \mathbb{Z}\backslash A(k,t)\right)  ,\\i=1,2,...,m}}\left\vert
q_{s,i,n-v}\right\vert \left\vert \left(  Q\Psi,\varphi_{v,s,t}\right)
\right\vert .
\]
Now using Schwards inequality of the space $l_{2}$ and taking (17) into
account we obtain
\begin{equation}
\left\vert S_{2}\right\vert \leq c_{5}\frac{1}{|k|}. \tag{27}%
\end{equation}
Thus (23) follows from (26), (27) and (19).
\end{proof}

Now we are ready to consider the Bloch eigenvalues in detail. The following
theorem follows from (20) and Lemma 1.

\begin{theorem}
If (21) holds, then the eigenvalues of the operator $L_{t}(Q)$ lying in
$D(\left(  2k+t)^{2},c\right)  $ are contained in $\varepsilon_{k}$
neighborhood $D\left(  \mu_{n,j},\varepsilon_{k}\right)  $ of the eigenvalues
$\mu_{n,j}(t)$ of $L_{t}(A)$ for $j=1,2,...,m$ and $n\in A(k,t),$ where
$\varepsilon_{k}\leq c_{6}(\mid\frac{1}{k}\mid+q_{k})^{1/r_{j}}.$
\end{theorem}

\begin{proof}
Dividing (20) by $(\Psi,\Phi_{n,j,s,r}^{\ast}),$ and using (22) and (23) we
obtain
\[
(\lambda-\mu_{n,i}(t))^{r+1}=\sum_{q=0}^{r}(\lambda-\mu_{n,i}(t))^{q}O\left(
\frac{1}{|k|}+q_{k}\right)  ,
\]
where $r+1\leq r_{i}$, from which we obtain the proof of the theorem.
\end{proof}

Now using these results we classify the spectrum of $L(Q)$. Theorems 2 and 3
imply that the spectrum of $L(Q)$ is asymptotically close to the spectrum of
the operator $L(A),$ since $\varepsilon_{k}\rightarrow0$ as $k\rightarrow
\infty.$ It is clear that if the Fourier coefficients $q_{i,j,k}$ of the
entries $q_{i,j}$ of $Q$ is $O(1/k),$ for example if the entries are the
continuous functions, then $\varepsilon_{k}=O(1/k).$ The following theorem
immediately follows from Theorems 2 and 3.

\begin{theorem}
If the matrix $A$ has no real eigenvalues, then the real component of the
spectrum of $L(Q)$ is contained in a finite interval $[a,b].$
\end{theorem}

\begin{proof}
\ If the eigenvalues $\mu_{1},\mu_{2},...,\mu_{p}$ of the matrix $A$ are
nonreal numbers then there exists $c_{7}$ such that
\begin{equation}
\varepsilon_{k}<\min_{j=1,2,...p}\left\vert \operatorname{Im}\mu
_{j}\right\vert \tag{28}%
\end{equation}
and (21) holds for $\left\vert k\right\vert >c_{7}.$ Inequality (28) implies
that the disks $D\left(  \mu_{n,j}(t),\varepsilon_{k}\right)  $ for
$\left\vert k\right\vert >c_{7},$ $n\in\mathbb{Z}$ and $j=1,2,...,p$ have no
intersection points with the real lines. Therefore using Theorems 2 and 3 we
obtain that the real component of the spectrum of $L(Q)$ is contained in the
bounded set $%
{\textstyle\bigcup\limits_{\left\vert k\right\vert \leq c_{7}}}
D(\left(  2k+t)^{2},c\right)  .$ It implies the proof of the theorem.
\end{proof}

Now we consider the cases when the matrix $A$ has the real eigenvalues and
investigate the real component of the spectrum of $L(q)$. Recall that (see the
end of the introduction) the real and nonreal distinct eigenvalues are denoted
respectively by $\mu_{1}<\mu_{2}<...<\mu_{s}$ and$\ \mu_{s+1},\mu
_{s+2},...,\mu_{p}$. To investigate the real spectrum of $L(Q)$, we use
Theorems 2 and 3 and find the conditions on $k$ and $t$ such that the boundary
of the closed disk $\overline{D\left(  \mu_{k,j}(t),\varepsilon_{k}\right)  }$
for $j\leq s$ belong to the resolvent set of the operator $L_{t}(Q).$ Since
$\mu_{k,i}(t)=\mu_{-k,i}(-t)$ (see (18)), it is enough to study the disks
$\overline{D\left(  \mu_{k,i}(t),\varepsilon_{k}\right)  }$ for $k\geq0$ and
$t\in(-1,1].$ We consider the closed disks $\overline{D\left(  \mu
_{n,j}(t),\varepsilon_{k}\right)  }$ for $j\leq s$, $n\in A(k,t)$ and $k\geq
N_{1},$ where $N_{1}$ is a positive integer such that if $k>N_{1},$ then
\begin{align}
\text{ }\delta_{k}  &  <c,\tag{29}\\
\frac{4}{3}\left(  2k-1\right)   &  >3c+\mid\mu_{j}\mid,\tag{30}\\
\delta_{k}  &  <\min_{j=s+1,s+2,...p}\left\vert \operatorname{Im}\mu
_{j}\right\vert ,\text{ }\delta_{k}<\mid\mu_{j}-\mu_{i}\mid\tag{31}%
\end{align}
for $j\leq s$, $i\leq s$ and $i\neq j$, where $\delta_{k}=2\max\left\{
\varepsilon_{k},\varepsilon_{-k},\varepsilon_{-k-1},\varepsilon_{-k+1}%
\right\}  $

\begin{lemma}
Suppose that $k>N_{1}$, $1\leq j\leq s$ and $t\in(-1,1].$ Then

$(a)$ The closed disk $\overline{D\left(  \mu_{k,j}(t),\varepsilon_{k}\right)
}$ has no common points with the disks $\overline{D(\left(  2n+t)^{2}%
,c\right)  }$ and $\overline{D\left(  \mu_{k,i}(t),\varepsilon_{k}\right)
}\ $for $n\notin A(k,t)$ and $i\neq j$.

$(b)$ The disk $\overline{D\left(  \mu_{k,j}(t),\varepsilon_{k}\right)  }$ has
no common points with the disks $\overline{D\left(  \mu_{n,i}(t),\varepsilon
_{n}\right)  }$ for $n\in A(k,t)$ and $s+1\leq i\leq p$ for all $t\in(-1,1].$
The disk $\overline{D\left(  \mu_{k,j}(t),\varepsilon_{k}\right)  }$ has no
common points with the disks $\overline{D\left(  \mu_{n,i}(t),\varepsilon
_{n}\right)  }$ for $n\in A(k,t)$ and $1\leq i\leq s$ if%
\begin{equation}
t\in\left(  (-1,1]\backslash U(j,k,\delta_{k})\right)  , \tag{32}%
\end{equation}
where $U(j,k,\delta_{k})=U(j,k,i,-k,\delta_{k})\cup U(j,k,i,-k-1,\delta
_{k})\cup U(j,k,i,-k+1,\delta_{k}),$%
\[
U(j,k,i,-k,\delta_{k})=\left(  \frac{\mu_{i}-\mu_{j}-\delta_{k}}{8k},\frac
{\mu_{i}-\mu_{j}+\delta_{k}}{8k}\right)  ,
\]%
\[
U(j,k,i,-k-1,\delta_{k})=\left(  1+\frac{\mu_{i}-\mu_{j}-\delta_{k}}%
{4(2k+1)},1+\frac{\mu_{i}-\mu_{j}+\delta_{k}}{4(2k+1)}\right)
\]
and%
\[
U(j,k,i,-k+1,\delta_{k})=\left(  -1+\frac{\mu_{i}-\mu_{j}-\delta_{k}}%
{4(2k+1)},-1+\frac{\mu_{i}-\mu_{j}+\delta_{k}}{4(2k+1)}\right)  .
\]

\end{lemma}

\begin{proof}
$(a)$ Using Proposition 1 and (30) one can easily verify that the distance
between the centres of the disks $\overline{D\left(  \mu_{k,j}(t),\varepsilon
_{k}\right)  }$ and $\overline{D(\left(  2n+t)^{2},c\right)  }$ for $n\notin
A(k,t)$ is greater than $3c.$ On the other hand, by (29) the total sum of the
radii of these disks is less than $2c.$ Therefore these disks have no common points.

It follows from (31) that the distance $\mid\mu_{j}-\mu_{i}\mid$ between the
centres of the disks $\overline{D\left(  \mu_{k,j}(t),\varepsilon_{k}\right)
}$ $\overline{D\left(  \mu_{k,i}(t),\varepsilon_{k}\right)  }$ is greater than
the total sum of their radii. That is why they also have no common point.

$(b)$ The proof of the first statement follows from the first inequality of
(31). Now we prove the second statement. By the definition of $A(k,t)$ it is
enough to prove that the disks $\overline{D\left(  \mu_{k,j}(t),\varepsilon
_{k}\right)  }$ and $\overline{D\left(  \mu_{n,i}(t),\varepsilon_{n}\right)
}$ have no common points for $n=-k,$ $n=-k-1$ and $n=-k+1$ respectively if
$t\in\left(  (-1,1]\backslash U(j,k,i,-k,\delta_{k})\right)  $, $t\in\left(
(-1,1]\backslash U(j,k,i,-k-1,\delta_{k})\right)  $ and $t\in\left(
(-1,1]\backslash U(j,k,i,-k+1,\delta_{k})\right)  ,$ where $i=1,2,...,s.$ This
can be easily verified by the direct calculations of the differences
$\mu_{k,j}(t)-\mu_{-k,i}(t)$, $\mu_{k,j}(t)-\mu_{-k-1,i}(t)$ and $\mu
_{k,j}(t)-\mu_{-k+1,i}(t)$ and using the definitions of the sets

$U(j,k,i,-k,\delta_{k}),$ $U(j,k,i,-k-1,\delta_{k})$ and $U(j,k,i,-k+1,\delta
_{k}).$ The lemma is proved.
\end{proof}

Now using this lemma we prove the following theorem which plays a crucial role
in the investigation of the real component of the spectrum of $L(Q).$

\begin{theorem}
Suppose that $k>N_{1}$, $j\leq s$ and (32) hold.

$(a)$ If $\mu_{j}$ is an eigenvalue of $A$ of multiplicity $v$ then the number
of the eigenvalues (counting multiplicity) of $L_{t}\left(  Q\right)  $ lying
in $D\left(  \mu_{k,j}(t),\varepsilon_{k}\right)  $ is $v$. Moreover, if
$\lambda$ is an eigenvalue of multiplicity $v$ of the operator $L_{t}(Q)$
lying in $D\left(  \mu_{k,j},\varepsilon_{k}\right)  ,$ then $\overline
{\lambda}$ is also an eigenvalue of the same multiplicity of $L_{t}(Q)$ lying
in $D\left(  \overline{\mu_{k,j}},\varepsilon_{k}\right)  .$

$(b)$ If $\mu_{j}$ is a simple eigenvalue of $A$, then the eigenvalue of
$L_{t}\left(  Q\right)  $ lying in $U\left(  \mu_{k,j},\varepsilon_{k}\right)
$ is a simple and real eigenvalue.
\end{theorem}

\begin{proof}
$(a)$ First we prove that the boundary of $D\left(  \mu_{k,j},\varepsilon
_{k}\right)  $ lies in the resolvent set of \ $L_{t}\left(  Q\right)  .$ By
Theorem 2 the eigenvalues of $L_{t}\left(  Q\right)  $ lie in the disks
$D(\left(  2n+t)^{2},c\right)  $ for $n\in\mathbb{Z}.$ By Lemma 2$(a)$ the
eigenvalues lying $D(\left(  2n+t)^{2},c\right)  $ for $n\in\left(
\mathbb{Z}\backslash A(k,t)\right)  $ do not lie in the boundary of $D\left(
\mu_{k,j},\varepsilon_{k}\right)  .$ It remains to prove the following
statement. The eigenvalues lying $D(\left(  2n+t)^{2},c\right)  $ for $n\in
A(k,t)$ do not lie in the boundary of $D\left(  \mu_{k,j},\varepsilon
_{k}\right)  .$ To prove this statement we consider the following three cases
separately: \ $t\in\left(  (-2/3,-1/3)\cup(1/3,2/3)\right)  ,$ $t\in
(-1/3,1/3)$, $t\in(2/3,1]$ and $t\in(-1,-2/3).$

If $t\in\left(  (-2/3,-1/3)\cup(1/3,2/3)\right)  $ then $A(k,t)=\left\{
k\right\}  $ (see definition of $A(k,t)$ in Proposition1). Therefore the proof
of the statement follows from Lemma 2$(a).$

If $t\in(-1/3,1/3),$ then $A(k,t)=\left\{  \pm k\right\}  .$ The case $n=k$
follows from Lemma 2$(a).$ Let us consider the case $n=-k.$ By Theorem 3 the
eigenvalues of $L_{t}\left(  Q\right)  $ lying in the disks $D(\left(
-2k+t)^{2},c\right)  $ are contained in $D\left(  \mu_{n,j},\varepsilon
_{-k}\right)  $ for $j=1,2,...,p$ and $n\in A(-k,t).$ Since $A(-k,t)=\left\{
\pm k\right\}  $ we need to consider the disks \ $D\left(  \mu_{k,j}%
,\varepsilon_{-k}\right)  $ and $D\left(  \mu_{k,j},\varepsilon_{-k}\right)
.$ The proof of the statement for $D\left(  \mu_{k,j},\varepsilon_{-k}\right)
$ and $D\left(  \mu_{-k,j},\varepsilon_{-k}\right)  $ follows from Lemma
2$(a)$ and Lemma 2$(b)$ respectively.

If $t\in(2/3,1]$ and $t\in(-1,-2/3)$ then $A(k,t)=\left\{  k,-k-1\right\}  $
and $A(k,t)=\left\{  k,-k+1\right\}  $ respectively. Moreover
$A(-k-1,t)=\left\{  -k-1,k\right\}  $ and $A(-k+1,t)=\left\{  -k+1,k\right\}
.$ Therefore repeating the proof of the case $t\in(-1/3,1/3)$ we get the proof
of this statement for the cases $t\in(2/3,1]$ and $t\in(-1,-2/3).$

Thus the circle $\left\{  \lambda\in\mathbb{C}:\left\vert \lambda-\mu
_{k,j}\right\vert =\varepsilon_{k}\right\}  $ lies in the resolvent set of
\ $L_{t}\left(  Q\right)  .$ Consider the following family of operators
\[
L_{\varepsilon}=L_{t}(A)+\varepsilon(Q-A),\text{ }0\leq\varepsilon\leq1.
\]
Repeating the proof of the case $\varepsilon=1,$ one can easily verify that
the circle $\left\{  \lambda\in\mathbb{C}:\left\vert \lambda-\mu
_{k,j}\right\vert =\varepsilon_{k}\right\}  $ lies in the resolvent set of
$L_{\varepsilon}$ for $\varepsilon\in\lbrack0,1].$ Therefore, taking into
account that the family $L_{\varepsilon}$ is halomorphic (in the sense of [2])
with respect to $\varepsilon,$ we obtain that the number of the eigenvalues
of\ $L_{\varepsilon}$ lying inside of $\left\{  \lambda\in\mathbb{C}%
:\left\vert \lambda-\mu_{k,j}\right\vert =\varepsilon_{k}\right\}  $ are the
same for all $\varepsilon\in\lbrack0,1]$. Since the operator $L_{0}=L_{t}(A)$
has $v$ eigenvalues (counting the multiplicity) inside $\left\{  \lambda
\in\mathbb{C}:\left\vert \lambda-\mu_{k,j}\right\vert =\varepsilon
_{k}\right\}  $ the operator $L_{t}(Q)$ has also $v$ eigenvalues inside of
this circle. Since $\mu_{j}$ for $j\leq s$ is a real number, if$\ \lambda\in
D\left(  \mu_{k,j},\varepsilon_{k}\right)  $ then $\overline{\lambda}$ $\in
D\left(  \mu_{k,j},\varepsilon_{k}\right)  $ too. Therefore, the last
statement of $(a)$ follows from Theorem 1.

$(b)$ It follows from $(a)$\ that if $\mu_{j}$ is a simple eigenvalue of $A$,
then the operator $L_{t}(Q)$ has a unique eigenvalue (counting multiplicity)
in the disk $D(\mu_{k,1}(t),\varepsilon_{k}).$ It means that $\lambda
_{k,j}(t)$ is a simple eigenvalue. Moreover, if $\mu_{j}$ is a real number and
$\lambda_{k,j}(t)$ is a nonreal eigenvalue of $L_{t}\left(  Q\right)  $ lying
in $D\left(  \mu_{k,j}(t),\varepsilon_{k}\right)  $ then $\overline
{\lambda_{k,j}(t)}$ is also an eigenvalue of $L_{t}\left(  Q\right)  $ lying
in $D\left(  \mu_{k,j}(t),\varepsilon_{k}\right)  ,$ which contradicts to the
above uniqueness. Note also that if $\mu_{j}$ is a simple eigenvalue of $A,$
then $r_{j}=1$ and hence $\varepsilon_{k}\leq c_{6}(\mid\frac{1}{k}\mid
+q_{k})$ (see Theorem 3). The theorem is proved.
\end{proof}

Now using Theorem 5 we investigate the real component of the spectrum of
$L(Q).$ This investigation is based on the following idea. First \ we consider
the set
\begin{equation}
\left\{  \mu_{k,j}(t):t\in\left(  (-1,1]\backslash U(j,k,\delta_{k})\right)
\right\}  , \tag{33}%
\end{equation}
where $U(j,k,\delta_{k})$ is defined in (32). By the definition of $\mu
_{k,j}(t)$ the set

$\left\{  \mu_{k,j}(t):t\in(-1,1]\right\}  $ is the interval $(\mu
_{j}+(2k-1)^{2},\mu_{j}+(2k+1)^{2}].$ The union $\left\{  \mu_{k,j}%
(t):k>N_{1},t\in(-1,1]\right\}  $ of these interval for \ $k>N_{1}$ is a half
line $(\mu_{j}+(2N_{1}+1)^{2},\infty).$ Since the set $U(j,k,\delta_{k})$ is
the union of the small intervals for large $k,$ the union of the sets (33) for
\ $k>N_{1}$ contains the main part of the half line $(\mu_{j}+(2N_{1}%
+1)^{2},\infty)$. Note that $U(j,k,\delta_{k})$ consist of the intervals
defined in Lemma 2$(b).$ Since $\delta_{k}\rightarrow0$ as $k\rightarrow
\infty,$ there exists $N(j)>N_{1}$ such that if $k\geq N(j)$, then these
intervals are pairwise disjoint. Thus we have
\begin{equation}
(-1,1]\backslash U(j,k,\delta_{k})=[a_{1},b_{1}]\cup\lbrack a_{2},b_{2}%
]\cup...\cup\lbrack a_{l},b_{l}], \tag{34}%
\end{equation}
for $k\geq N(j),$ where $a_{1}<b_{1}<a_{2}<b_{2}<...$ and the sum of the
length of these intervals is asymptotically close to the length of $(-1,1].$
Namely
\begin{equation}
2-%
{\textstyle\sum\limits_{i=1,2,...l}}
\left(  b_{i}-a_{i}\right)  <c_{9}\frac{\delta_{k}}{k}. \tag{35}%
\end{equation}
Note that if $a_{1}=-1,$ then in (34) the interval $[a_{1},b_{1}]$ should be
replaced by $(a_{1},b_{1}].$ This replacement does not change the
investigation. Therefore, without loss of generality, we assume that (34)
holds. Since $\mu_{k,j}$ is an increasing function on $(-1,1],$ the image
$\mu_{k,j}([a_{i},b_{i}])$ of the interval $[a_{i},b_{i}]$ is the interval
$[\mu_{k,j}(a_{i}),\mu_{k,j}(b_{i})].$ Therefore
\begin{equation}
\mu_{k,j}\left(  (-1,1]\backslash U(j,k,\delta_{k})\right)  =%
{\textstyle\bigcup\nolimits_{i=1,2,...,l}}
[\mu_{k,j}(a_{i}),\mu_{k,j}(b_{i})] \tag{36}%
\end{equation}
for $k\geq N(j).$ Using (35) and definition of the function $\mu_{k,j}$ we see
that the length of the intervals in (36) is asymptotically close to the length
of $(\mu_{j}+(2k-1)^{2},\mu_{j}+(2k+1)^{2}]$ in the sense that
\begin{equation}
8k-%
{\textstyle\sum_{i=1,2,...l}}
\left(  \mu_{k,j}\left(  b_{i})-\mu_{k,j}(a_{i}\right)  \right)  <c_{10}%
\delta_{k}\rightarrow0 \tag{37}%
\end{equation}
as $k\rightarrow\infty.$ Thus the set $%
{\textstyle\bigcup\limits_{k\geq N(j)}}
\mu_{k,j}\left(  (-1,1]\backslash U(j,k,\delta_{k})\right)  $ contains the
main part of the half line $(\mu_{j}+(2N(j)-1)^{2},\infty).$

Now, using Theorem 5 and (37) we prove that the spectrum of $L(Q)$ contains
the main part of the half line $[0,\infty)$. First let us explain it in the
simplest case, when $\mu_{j}$ is a simple eigenvalue of $A.$ Then by Theorem
5$(b)$ the eigenvalue $\lambda_{k,j}(t)$ of $L_{t}\left(  Q\right)  $ lying in
$U\left(  \mu_{k,j},\varepsilon_{k}\right)  $ is a simple and real eigenvalue.
Moreover, the simplicity of the eigenvalue $\lambda_{k,j}(t)$ implies that it
continuously depend on $t.$ As a result the set
\begin{equation}
\left\{  \lambda_{k,j}(t):t\in\left(  (-1,1]\backslash U(j,k,\delta
_{k})\right)  \right\}  \tag{38}%
\end{equation}
contains the main part of the set (33). Therefore the union of the sets (38)
for $k\geq N(j)$ contain the main part of the half line $(\mu_{j}%
+(2N(j)-1)^{2},\infty)$.

Now let us discuss it in more general case when $\mu_{j}$ is a real eigenvalue
of odd multiplicity $v$. Then by Theorem 5$(a)$ the operator $L_{t}(Q)$ has
$v$ eigenvalues in the neighborhood of $\mu_{k,j}(t).$ All these eigenvalues
can not be nonreal, since they are pairwise conjugate number and $v$ is an odd
number. Thus the operator $L_{t}(Q)$ has a real eigenvalue. In the following
theorem we prove that these real eigenvalues fill the main part of the half
line $[0,\infty).$

\begin{theorem}
Suppose that the matrix $A$ has a real eigenvalue $\mu_{j}$ of odd multiplicity.

$(a)$ Then the real component of the spectrum of $L(Q)$ contains the
subintervals
\begin{equation}
\Gamma_{k,j,i}=[\mu_{k,j}(a_{i})+\delta_{k},\mu_{k,j}(b_{i})-\delta_{k}]
\tag{39}%
\end{equation}
for $i=1,2,...l$ of the interval $\Gamma_{k,j}=(\mu_{j}+(2k-1)^{2},\mu
_{j}+(2k+1)^{2}]$, where $k\geq N(j).$ These subintervals are pairwise
disjoint and satisfy the inequality
\begin{equation}
\mu(\Gamma_{k,j})-\sum_{i=1,2,...l}\mu(\Gamma_{k,j,i})<c_{11}\delta_{k},
\tag{40}%
\end{equation}
where $\mu(E)$ denotes the measure of the set $E$ and $\delta_{k}=O(\mid
\frac{1}{k}\mid+q_{k})^{1/r_{j}}$

$(b)$ The spectrum of $L(Q)$ contains the main part of $[0,\infty)$ in the
sense that \
\begin{equation}
\lim_{n\rightarrow\infty}\frac{\mu([0,n]\backslash\sigma(L(Q)))}{\mu
(\sigma(L(Q))\cap\lbrack0,n])}=0. \tag{41}%
\end{equation}

\end{theorem}

\begin{proof}
Let $\mu_{j}$ be a real eigenvalue of multiplicity $v$ of the matrix $A,$
where $v$ is an odd number. Then by Theorem 5 the disk $D\left(  \mu
_{k,j}(t),\delta_{k}\right)  $ for $t\in\lbrack a_{i},b_{i}]$ and $i=1,2,...l$
contains $v$ eigenvalues (counting the multiplicity) of $L_{t}(Q),$ where
$[a_{i},b_{i}]$ is defined in (34). Let us denote they by $\lambda
_{k,1}(t),\lambda_{k,2}(t),...,\lambda_{k,v}(t).$ Consider the unordered
$v$-tuple $\left\{  \lambda_{k,1}(t),\lambda_{k,2}(t),...,\lambda
_{k,v}(t)\right\}  .$ It follows from (4) that this unordered $v$-tuple depend
continuously (in sense of Theorem 5.2 of [2])\ on the parameter $t\in\lbrack
a_{i},b_{i}].$ Then by Theorem 5.2 of [2] there exist $p$ single-valued
continuous functions $\lambda_{1}(t),\lambda_{2}(t),...,\lambda_{v}(t)$ the
value of which constitute the $v$-tuple $\left\{  \lambda_{k,1}(t),\lambda
_{k,2}(t),...,\lambda_{k,v}(t)\right\}  $ for $t\in\lbrack a_{i},b_{i}].$

Now we prove that
\begin{equation}
\lbrack\mu_{k,j}(a_{i})+\delta_{k},\mu_{k,j}(b_{i})-\delta_{k}]\subset\left(
\cup_{s=1}^{v}\gamma_{s}\right)  \subset\sigma(L(Q)), \tag{42}%
\end{equation}
where $\gamma_{s}$ is the curve $\left\{  \lambda_{s}(t):t\in\lbrack
a_{i},b_{i}]\right\}  .$ Assume the converse. Then there exists
\[
a\in\lbrack\mu_{k,j}(a_{i})+\delta_{k},\mu_{k,j}(b_{i})-\delta_{k}%
]\backslash\left(  \cup_{s=1}^{v}\gamma_{s}\right)  .
\]
It means that the continuous curves $\gamma_{s}$ extend from $\lambda
_{s}(a_{i})\in U\left(  \mu_{k,j}(a_{i}),\delta_{k}\right)  $ $\ $to
$\lambda_{s}(b_{i})\in U\left(  \mu_{k,j}(b_{i}),\delta_{k}\right)  $ pas
above or below of the point $a,$ since $\operatorname{Re}\lambda_{s}%
(a_{i})\ <a<\operatorname{Re}\lambda_{s}(b_{i}).$ Moreover by Theorem 5$(a)$
if $\gamma_{s}$ passes above of $a$ then there exist $j\in\left\{
1,2,...,v\right\}  $ such that $j\neq s$ and $\gamma_{j}$ passes below
of$\ a.$ It implies that the number $v$\ of the curves $\gamma_{s}$ is an even
number. It contradicts to the assumption that $v$ is an odd number. Thus (42)
is proved and the subintervals $\Gamma_{k,j,i}$ defined in (39) are the
subsets of the spectrum of $L(Q)$. Estimation (40) follows from (37). Note
that $N(j)$ can be chosen so that $\mu_{k,j}\left(  b_{i})-\mu_{k,j}%
(a_{i}\right)  -2\delta_{k}>0$ for $k\geq N(j).$

The proof of (41) follows from $(a).$
\end{proof}

\begin{corollary}
If $m$ is an odd number, then the real component of the spectrum of $L(Q)$
contains the main part of $[0,\infty)$ and (41) holds.
\end{corollary}

Now we find a condition on the eigenvalues of the matrix $A$ for which the the
real component of the spectrum of $L(Q)$ contains a half line $[H,\infty)$ for
some $H.$

\begin{theorem}
If the matrix $A$ has at least three real eigenvalues $\mu_{j_{1}},$
$\mu_{j_{2}}$ and $\mu_{j_{3}}$ of odd multiplicities such that
\begin{equation}
\min_{i_{1},i_{2},i_{3}}\left(  diam(\{\mu_{j_{1}}+\mu_{i_{1}},\mu_{j_{2}}%
+\mu_{i_{2}},\mu_{j_{3}}+\mu_{i_{3}}\})\right)  =d\neq0, \tag{43}%
\end{equation}
where $i_{k}=1,2,...,s$ for $k=1,2,3$ and
\[
diam(E)=\sup_{x,y\in E}\mid x-y\mid,
\]
then there exists a number $H$ such that $[H,\infty)\in\sigma(L(Q))$.
\end{theorem}

\begin{proof}
By Theorem 6$(a)$ we have
\begin{equation}
\sigma(L(Q))\supset\left(  S(j_{1},N(j_{1}))\cup S(j_{2},N(j_{2}))\cup
S(j_{3},N(j_{3}))\right)  , \tag{44}%
\end{equation}
where
\begin{equation}
S(j,N)=%
{\textstyle\bigcup\limits_{i=1,2,...l;\text{ }k\geq N}}
[\mu_{k,j}(a_{i})+\delta_{k},\mu_{k,j}(b_{i})-\delta_{k}] \tag{45}%
\end{equation}
Using the definition of $U(j,k,\delta_{k})$\ one can easily verify that there
exists $\alpha_{k}$ such that $\alpha_{k}\rightarrow0$ as $k\rightarrow\infty$
and
\begin{equation}
\mu_{k,j}(U(j,k,\delta_{k}))\subset%
{\textstyle\bigcup\limits_{n=-1,0,1;\text{ }i=1,2,...,s}}
C(k,j,i,\alpha_{k},n), \tag{46}%
\end{equation}
where
\begin{equation}
C(k,j,i,a,n)=\{x\in\mathbb{R}:\mid x-((2k+n))^{2}-\frac{\mu_{i}+\mu_{j}}%
{2}\mid<a\}. \tag{47}%
\end{equation}
Then we have
\[%
{\textstyle\bigcup\limits_{k\geq N(j_{p})}}
\{\mu_{k,j_{p}}(t):t\in(-1,1]\backslash U(j_{p},k,\alpha_{k})\}\supset
(h,\infty)\backslash%
{\textstyle\bigcup\limits_{\substack{n=-1,0,1;\\k\geq N\\i=1,2,...,s}}}
C(k,j_{p},i,\alpha_{k},n),
\]
where $h=(2N(j_{p})-1)^{2}+\mu_{j_{p}}$ and $p=1,2,3.$ Therefore using
(44)-(47) we obtain that there exist $H>h$ and $\beta_{k}$ such that
$\beta_{k}\rightarrow0$ as $k\rightarrow\infty$ and
\begin{equation}
\sigma(L(q))\supset(H,\infty)\backslash%
{\textstyle\bigcup\limits_{n=-1,0,1;\text{ }k\geq N;\text{ }i=1,2,...,s}}
C(k,j_{p},i,\beta_{k},n) \tag{48}%
\end{equation}
{} for all $p=1,2,3,$ where $N\geq\max\left\{  N(j_{1}),N(j_{2}),N(j_{3}%
)\right\}  .$ Now we prove that
\begin{equation}
\lbrack H,\infty)\subset\sigma(L(q)). \tag{49}%
\end{equation}
By (48) to prove (49) it is enough to show that the set
\begin{equation}%
{\textstyle\bigcap\limits_{p=1,2,3}}
\left(
{\textstyle\bigcup\limits_{n=-1,0,1;\text{ }k\geq N;\text{ }i=1,2,...,s}}
C(k,j_{p},i,\beta_{k},n)\right)  \tag{50}%
\end{equation}
is empty. Since $\beta_{k}\rightarrow0$ the number $N$ can be chosen so that
$4\beta_{k}<d$ for $k\geq N$. If the set (50) contains an element $x,$ then%
\[
x\in%
{\textstyle\bigcup\limits_{n=-1,0,1;\text{ }k\geq N;\text{ }i=1,2,...,s}}
C(k,j_{p},i,\beta_{k},n)
\]
for all $p=1,2,3.$ Using this and the definition of $C(k,j_{p},i,a,n)$ (see
(47)), we obtain that there exist $k\geq N;$ $n\in\left\{  -1,0,1\right\}  $
and $i_{p}\in\left\{  1,2,...,s\right\}  $ such that
\[
\mid x-(\pi(2k+n))^{2}-\frac{\mu_{j_{p}}+\mu_{i_{p}}}{2}\mid<\beta_{k}%
\]
for all $p=1,2,3$. This implies that
\[
\left\vert \left(  \mu_{j_{q}}+\mu_{i_{q}}\right)  -\left(  \mu_{j_{p}}%
+\mu_{i_{p}}\right)  \right\vert <4\beta_{k}<d
\]
for all $p,q=1,2,3$, where $p\neq q.$ It contradicts the condition (43).
Theorem is proved.
\end{proof}

Note that in paper [10] we proved Theorem 7 for the self-adjoint operator
$L(Q)$ under additional condition that $\mu_{j_{1}},$ $\mu_{j_{2}}$ and
$\mu_{j_{3}}$ are the simple eigenvalues. In this paper we prove it for the
non-self-adjoint operator $L(Q)$ and without the simplicity condition.

\section{Data Availability}

Data sharing is not applicable to this article as no new data were created or
analyzed in this study.


\begin{thebibliography}{99}                                                                                               %


\bibitem {}M. S. P. Eastham, \textit{The Spectral Theory of Periodic
Differential Operators} (Hafner, New York, 1974).

\bibitem {}T. Kato, \textit{Perturbation Theory for Linear Operators}
(Springer-Verlag, Berlin, 1980).

\bibitem {}K. G. Makris, R. El-Ganainy, D. N. Christodoulides, and Z. H.
Musslimani, "PT-Symmetric Periodic Optical Potentials", \textit{Int. J. Theor.
Phys}. \textbf{50,} 1019--1041 (2011).

\bibitem {}F. G. Maksudov and O. A. Veliev, "Spectral Analysis of Differential
Operators with Periodic Matrix Coefficients", \textit{Differ. Equations},
\textbf{25}, 271-277 (1989).

\bibitem {}D. C. McGarvey, "Differential operators with periodic coefficients
in $L_{p}(-\infty,\infty)"$, \textit{J. Math. Anal. Appl.} \textbf{11},
564-596 (1965).

\bibitem {}M. A. Naimark,\textit{ Linear Differential Operators} (George G.
Harap\&Company, London, 1967).

\bibitem {}F. S. Rofe-Beketov, "The spectrum of nonselfadjoint differential
operators with periodic coefficients", \textit{Soviet Math. Dokl.} \textbf{4},
1563-1566 (1963).

\bibitem {}K. C. Shin, "On the shape of spectra for non-self-adjoint periodic
Schr%
\"{}%
odinger operators", \textit{J. Phys. A: Math. Gen.} \textbf{37,} 8287--8291 (2004).

\bibitem {}O. A. Veliev, "On the non-self-adjoint Sturm-Liouville operators
with matrix potentials", \textit{Mathematical Notes}, \textbf{81}, 440-448 (2007).

\bibitem {}O. A. Veliev , "On the Hill's operator with a matrix potential",
\textit{Mathematische Nachrichten}, \textbf{281}, 1341-1350 (2008).

\bibitem {}O. A. Veliev, "Uniform convergence of the spectral expansion for a
differential operator with periodic matrix coefficients", \textit{Boundary
Value Problems}, Article ID 628973, 22 pp. (2008).

\bibitem {}O. A. Veliev, \textit{Non-self-adjoint Schr\"{o}dinger operator
with a periodic potential}, (Springer Nature, Switzerland, 2021).
\end{thebibliography}
\end{document}